\documentclass[review,onefignum,onetabnum]{amsart}
\usepackage[T1]{fontenc}
\usepackage[utf8]{inputenc}
\usepackage[margin=1.2in]{geometry}
\usepackage{lmodern}
\usepackage[english]{babel}
\usepackage{graphicx}
\usepackage{amsmath}
\usepackage{amsfonts}
\usepackage{amssymb} 
\usepackage{bbm}
\usepackage{bm} 
\usepackage{lipsum}
\usepackage{xcolor}
\usepackage{url}
\usepackage{hyperref}
\usepackage{dsfont}

\allowdisplaybreaks
\newtheorem{thm}{Theorem}
\newtheorem{lem}[thm]{Lemma}
\newtheorem{prop}[thm]{Proposition}

\newtheorem{defn}[thm]{Definition}
\newtheorem{rem}[thm]{Remark}

\def \be {\begin{equation}}
\def \ee {\end{equation}}

\def \S{\Sigma}
\def \s{\sigma}
\def \a{\alpha}
\def \g{\gamma}
\def \ra{\rightarrow}
\def \bs {\boldsymbol}
\def \E {\mathbb{E}}
\def \P {\mathbb{P}}
\def \R {\mathbb{R}}
\def \N {\mathbb{N}}
\def\e {\varepsilon}
\def \ra {\rightarrow}

\newcommand{\ind}{\mathds{1}}

\begin{document}

\title{Long Range Games}

\author{Francesca Albertini}
\address{Dipartimento di Tecnica e Gestione dei Sistemi Industriali, Università di Padova, Stradella San Nicola 3, Vicenza, Italy}
\email{francesca.albertini@unipd.it}
\author{Paolo Dai Pra }
\address{Dipartimento di Informatica, Università di Verona,  Strada le Grazie 15, Verona, Italy}
\email{paolo.daipra@univr.it}


\date{\today}

\begin{abstract}
We consider $N$-player games, in continuous time, finite state space and finite time horizon, on a geometrical structure possessing a macroscopic limit in a suitable sense. This geometrical structure breaks the permutation invariance property that gives rise to mean field games. The corresponding limit game is a variant of mean field games that we call {\em long range game}. We prove that this asymptotic scheme satisfies the following key properties: a) the long range game admits al least one equilibrium; b) this equilibrium is unique under a suitable monotonicity condition; c) the feedback corresponding to any equilibrium of the long range game is a quasi-Nash equilibrium for the $N$-player games. We finally show that this scheme includes several examples of interaction mechanisms, in particular Kac-type interactions and interactions on generalized Erd\"{o}s-Renyi graphs. \\

\end{abstract}

\maketitle

\section{Introduction}

Modeling a stochastic system of $N$ interacting components denoted by $x_1,x_2,\ldots,x_N$, both in a static and in a dynamic setting, usually amounts to define interaction functions of the form
\[
F_i(x_i, {\bf x}^i)
\]
describing the effect of the interactions with all other components ${\bf x}^i := (x_j)_{j \neq i}$ on the state $x_i$ of the $i$-th component. In mean field systems one assumes these interaction functions are of the form
\begin{equation} \label{meanfield}
F_i(x_i, {\bf x}^i) \, = \, H(x_i, m^{N,i}_{{\bf x}}),
\end{equation}
where $m^{N,i}_{{\bf x}}$ is the {\em empirical measure}
\[
m^{N,i}_{{\bf x}} := \frac{1}{N-1}\sum_{j:j\neq i} \delta_{x_j},
\]
and $H$ is a given function, the same for all $i=1,2,\ldots,N$, defined on the product $\Sigma \times P(\Sigma)$, where $\Sigma$ is the state space of a single component, and $P(\Sigma)$ is the simplex of probability measures on $\Sigma$. Under suitable regularity conditions, these systems are characterized  by permutation invariance and the property that the effect of a single component on the distribution of the others is ``of order $\frac{1}{N}$'', see \cite{card10}, Theorem 2.1 for a precise statement.

A special but relevant example is when $H$ is of the form
\[
H(x,m) = \int f(x,y) m(dy),
\]
so that
\begin{equation} \label{mfpair}
F_i(x_i, {\bf x}^i) = \frac{1}{N-1} \sum_{j:j \neq i} f(x_i, x_j),
\end{equation}
in other words only {\em pair interactions} appear. 

A geometrical structure in mean field models can be introduced as in \cite{KUH63}, modifying \eqref{mfpair} as follows:
\begin{equation} \label{kac}
F_i(x_i, {\bf x}^i) = \frac{1}{N-1} \sum_{j:j \neq i} K\left(\frac{i}{N}, \frac{j}{N} \right) f(x_i, x_j),
\end{equation}
where $K:[0,1] \times [0,1] \ra \R$ is a weight function that brakes the permutation symmetry of the model. Interactions of this form are sometimes referred to as of {\em Kac type}, and have been extensively studies in several classical models is Statistical Mechanics, see e.g. \cite{BP98,COP93}. Note that the weight function $K$ provides a structure of {\em weighted graph} to the label set $\{1,2,\ldots,N\}$. Further extensions, attracting considerable attention in the last decade, include the case in which the function $K$ in \eqref{kac} is allowed to depend on $N$ explicitly, and to be {\em random}, as in the so-called {\em Graphon} mean-field models, see e.g. \cite{BCW23} for interacting diffusions with graphon interaction, and \cite{PO19,PO23, CCGL22, CH21} for games related models.

In this paper we propose a class of $N$-player games with a general form of interaction, which includes Kac-type and graphon-type interactions, and for which the limit as $N \ra +\infty$ can be appropriately dealt with. In order to emphasize the essence of this class of models, we minimize the technical difficulties by assuming that players have a finite state space, they evolve as continuous time Markov Chain, and the jump intensity coincides with player's control, which is constrained to be a Markovian feedback of the state of the whole community. The classical theory of mean field games (\cite{LL07, HMC06}) has been developed in details in this finite state space context in \cite{GMS13}. Unlike in mean field games, each player, besides her controlled state, is characterized by her {\em position} $u^N_i$ in a {\em position space} $\mathcal{U}$, that will be assumed to be a compact metric space. These positions are {\em fixed in time}; we require however that, as $N \ra +\infty$, the empirical position measure $\displaystyle{\frac{1}{N} \sum_{i=1}^N \delta_{u^N_i}}$ converges to a limit measure in weak topology. The interaction among players emerges through the cost function of each player, which depends on the states and positions through functions fo the form
\[
F^N(x_i, m^{N,i}_{{\bf x}}, u_i^N),
\]
but now 
\[
m^{N,i}_{{\bf x}} := \frac{1}{N-1}\sum_{j:j\neq i} \delta_{x_j,u_j^N}
\]
denotes the {\em joint} empirical measure of states and positions. Note that this includes pair interactions of the form
\[
\frac{1}{N-1} \sum_{j:j \neq i} K^N\left(u_i^N, u_j^N \right) f(x_i, x_j),
\]
in particular Kac and Graphon interactions. 

At a formal level, it will be easy to identify the limit, as $N \ra +\infty$, of the $N$-player game, that we call {\em long range game}. Unlike in mean field games, where a single representative player suffices to describe the behavior of the infinite community, here in the limit game we need to consider one player for each position $u \in \mathcal{U}$. However, the notion of {\em equilibrium} for the limit game can be defined similarly to that of standard mean field games. Our main aim is to show that the following (minimal) consistency properties hold:
\begin{itemize}
\item[a)]
Under standard convexity/regularity assumptions the limit game possesses at least one equilibrium.
\item[b)]
Under a suitable {\em monotonicity} condition, the equilibrium is unique.
\item[c)] 
The {\em Approximation property} holds: the feedback corresponding to {\em any} equilibrium of the limit game is {\em quasi-Nash} for the $N$-player game (see Definition \ref{def:eNash}).
\end{itemize}
Some aspect of these results are worth of a preliminary discussion. To begin with, points a) and b) above are conceptually straightforward. Indeed, the limit game we obtain can be interpreted as a mean field game in which: i) the state of the player is the pair state/position; ii) the position component of the state is static and random. This mean field game is neither finite state nor diffusive, so we checked that the usual fixed point - monotonicity argument apply, but this is essentially an exercise. The approximation property, on the other hand, is what we view as our main contribution. First of all, the above argument of extending the state with the position does not apply to the $N$-player game; indeed, we do not assume that the positions of the $N$ player has a permutation invariant distribution, and not even that it is random. Thus, the $N$-player game does not satisfy the usual symmetry conditions in Mean Field Game modeling. Moreover, as we will see later, for a given limit game there may be many different $N$-player games for which the approximation property holds. We succeed in providing general conditions for the validity of the approximation property, and test them in relevant examples. Our hope is that this general framework could provide both theoretical basis and inspiration for modeling games with many players whose interaction involves a geometric structure.

As mentioned above, we assume that the player's strategy coincides with the jump intensity. This allows to avoid two difficulties. The first is the use of {\em relaxed strategies}, needed to convexify the problem, see \cite{CF20} for the case of finite state mean field games. The second is concerned with the difficulties related with {\em propagation of chaos} in this position dependent context (see \cite{BCW23}). Finally, in the description of the long range game, we do not represent the dynamics of players in different position as a family of stochastic processes in the same probability space; this raises measurability issues that can be overcome with the so-called Fubini extension (see \cite{ACL22,CDP24}). All these issues have independent interest, but are not dealt with here.

This paper is organized as follows. In Sections 2 and 3 we model the $N$-player game, identify the limit long range game, and state existence of equilibrium and uniqueness under monotonicity. In Section 4 we deal with the approximation property. Section 5 is devoted to examples. Finally, in Section 6 we give the proofs of the results contained in Sections 3 and 4.

\section{The $N$-player game}

We consider the continuous time evolution of the states $X^N_i(t)$, $i=1,2,\ldots,N$, of $N$ players; the state of each player belongs to a given finite set $\S$ that we identify with the set $\{1,2, \ldots,d\}$, so $d = |\S|$. Each player is also characterized by her spatial position $u^N_i \in \mathcal{U}$, where $\mathcal{U}$ is a compact   metric space. We assume different players have different positions.
Players are allowed to control, via an arbitrary {\em feedback}, their jump rates. For $i=1,2,\ldots,N$ and $y \in \S$, we denote by $\a_y^i: [0,T] \times  \S^{N} \ra [0,+\infty)$ the rate at which player $i$ jumps to the state $y \in \S$; it is allowed to depend on the time $t \in [0,T]$, and on the state ${\bm x} = (x_i)_{i=1}^N$ of all players. Denoting by $A_N$ the set of functions $ [0,T] \times  \S^{N} \ra [0,+\infty)$ which are measurable and locally integrable in time, we assume $\a_y^i \in A_N$. So we write $\a^i \in \mathcal{A} := A_N^{\S}$, and let ${\bs{\a}}^N \in \mathcal{A}^N$ denote the controls of all players, that will be also called {\em strategy vector}.
In more rigorous terms, for ${\bs \a}^N \in \mathcal{A}^N$, the state evolution $\bm{X}_t := (X_i(t))_{i=1}^N$ is a Markov process, whose law is uniquely determined as solution to the martingale problem for the time-dependent generator
\[
\mathcal{L}_t f({\bm x}) = \sum_{i=1}^N \sum_{y \in \S} \a^i_y (t,{\bm x}) \left[ f([{\bm x}^i,y]) - f({\bm x})\right],
\]
where
\[
[{\bm x}^i,y]_j = \left\{ \begin{array}{ll} x_j & \mbox{for } j \neq i \\ y &  \mbox{for } j = i. \end{array} \right.
\]
Alternatively, the state evolution can be given as solution of the stochastic differential equations
\[
X_i^N(t) = X_i^N(0) +  \int_{[0,t]\times [0,+\infty) \times \S }\left( y - X_i^N(s^-) \right) \ind_{(0, \a^i_y (s^-,{\bm X^N(s^-)}))}(\xi) \mathcal{N}_i(ds,d\xi,dy),
\]
where $\left(\mathcal{N}_i\right)_{i\geq 1}$ are independent Poisson random measures on $[0,T] \times [0,+\infty) \times \S$ with intensity measure given by the product of the Lebesgue measure on $[0,T]$, the Lebesgue measure on $\R$ and the counting measure on $\S$.

Now let $P(\mathcal{U} \times \S)$ be the set of probability measures on the product of the Borel $\s$-field of $\mathcal{U}$ and the set of all subsets of $\S$, provided with the topology of weak convergence. With this topology $P(\mathcal{U} \times \S)$ is compact. To every ${\bm x} \in \S^N$ we associate 
the element of $P(\mathcal{U} \times \S)$
\be \label{empmeas}
m^{N,i}_{\bm{x}} :=  \frac{1}{N-1}\sum_{j=1, j \neq i}^N \delta_{(u^N_j,x_j)}.
\ee
Thus $m^{N,i}_{\bm{x}} $ is the {\em joint} empirical measure of position and state of all players except player $i$. Note that the dependence of $m^{N,i}_{\bm{x}} $ on $u_j^N$, $j \neq i$ is omitted, as the positions $u_j^N$ do not change in time. 
Given the functions
\[
L: \S \times [0,+\infty)^{\S} \times \mathcal{U} \ra \R, \ \ F_N: \S \times P(\mathcal{U} \times \S) \times \mathcal{U} \ra \R, \ \ G_N: \S \times P(\mathcal{U} \times \S) \times \mathcal{U}  \ra \R,
\]
the feedback controls ${\bs{\a}}^N \in \mathcal{A}^N$ and the corresponding process ${\bm X}$, the {\em cost} associated to the $i$-th player is given by
\begin{equation} \label{iplayercost}
\begin{split}
J_i^N(\boldsymbol{\alpha}^N) := \mathbb{E}\Bigg[\int_{0}^T \big[L(X_i(t), \alpha^i(t,\bm{X}_t),u_i^N)   + & F_N\big(X_i(t), m^{N,i}_{\bm{X}}(t), u_i^N \big) \big] dt  \\ & + G_N\big(X_i(T), m^{N,i}_{\bm{X}}(T), u_i^N \big) \Bigg].
\end{split}
\end{equation}
For a strategy vector $\boldsymbol{\alpha}^N = (\alpha^1, \dots, \alpha^N)\in\mathcal{A}^N$ and   $\beta\in\mathcal{A}$, denote by $[\boldsymbol{\alpha}^{N,-i}; \beta]$ the perturbed strategy vector given by
\begin{equation*}
[\boldsymbol{\alpha}^{N,-i}; \beta]_j := 
\begin{cases}
\alpha_j, \ \ j \neq i\\ 
\beta, \ \ j = i.
\end{cases}
\end{equation*}
\begin{defn} \label{def:Nash}
A strategy vector $\boldsymbol{\alpha}^N$ is a Nash equilibrium for the $N$-player game if for each $i=1, \dots, N$
\begin{equation*}
J_i^N(\boldsymbol{\alpha}^N) = \inf_{\beta\in\mathcal{A}}J_i^N([\boldsymbol{\alpha}^{N,-i}; \beta]).
\end{equation*}
\end{defn}
The notion of Nash equilibrium can be weakened as follows.
\begin{defn} \label{def:eNash}
For  given $\e_N, \delta_N \geq0$ , a strategy vector $\boldsymbol{\alpha}^N$ is a $\e_N$-$\delta_N$-Nash equilibrium for the $N$-player game if the inequality
\begin{equation} \label{epsNash}
J_i^N(\boldsymbol{\alpha}^N) \leq  \inf_{\beta\in\mathcal{A}}J_i^N([\boldsymbol{\alpha}^{N,-i}; \beta]) + 
\e_N
\end{equation}
holds for a number of players not smaller than $N(1-\delta_N)$. If $\delta_N=0$ we say that $\boldsymbol{\alpha}^N$ is a $\e_N$-Nash equilibrium. 
\end{defn}
\begin{rem} \label{rem:openloop}
In what follows we will actually use a slightly stronger notion of $\e_N$-$\delta_N$-Nash equilibrium, in which $\beta\in\mathcal{A}$ in \eqref{epsNash} is replaced by $\beta \in \mathcal{B}$, where $\mathcal{B}$ is the set of strategies that are predictable with respect to the filtration generated by the Poisson random measures $\mathcal{N}_i$, $i \geq 1$.
\end{rem}

\section{Long range game: existence and uniqueness}

The limit of the $N$-player game as $N \ra +\infty$ is based on the assumption that the empirical distribution of players' positions has a weak limit as $N \ra +\infty$: there exists a probability $\mu \in P(\mathcal{U})$ such that
\be \label{limpos}
\frac{1}{N} \sum_{i=1}^N \delta_{u_i^N} \ \ra \ \mu \ \ \mbox{weakly as $N \ra +\infty$}.
\ee
Moreover, we assume that the functions $F_N$ and $G_N$ appearing in the cost $J_i^N$ possess a limit, that we denote by $F$ and $G$, in a sense that will be specified later.

It is useful to introduce the notation
\[
\mathcal{P}_{\mu} = \{ m \in P(\mathcal{U} \times \S): \, m(A \times \S) = \mu(A) \mbox{ for all } A \mbox{ Borel subsets of } \mathcal{U}\}
\]
for the set of probability on $\mathcal{U} \times \S$ with first marginal $\mu$. 
Note that $\mathcal{P}_{\mu}$ is a compact subset of $P(\mathcal{U} \times \S)$.

The limiting game is defined as follows. For each $u \in \mathcal{U}$ a player with state $Y_u$ faces the following problem:
\begin{itemize}
\item[(i)]
the player controls its jump intensities $\a_y: [0,T] \times \S \ra [0,+\infty)$, $y \in \S$, via feedback controls depending on time and on her own state,  that are locally integrable in time; we denote by $Y_u^{\a}(t)$ the state at time $t$ with the feedback $\a$.
\item[(ii)]
For a given deterministic flow of probability measures 
$m : [0,T] \to \mathcal{P}_{\mu}$, the player aims at minimizing the cost
{\small
\be \label{mfcost}
J_u(\alpha,m):= \mathbb{E} \left[ \int_0^T
	\left[L(Y^{\a}_u(t), \alpha(t,Y^{\a}_u(t)), u) + F(Y^{\a}_u(t),m(t),u)\right]dt +G(Y^{\a}_u(T), m(T),u)\right].
\ee 
}
\item[(iii)]
Denote by $\a^{*,m}$ the optimal control for the above problem (next assumptions will imply existence and uniqueness), and let $(Y_u^{*,m}(t))_{t \in [0,T]}$ be the corresponding optimal process. The flow $(m(t))_{t \in [0,T]}$ must satisfy the following consistency relation:
\[
m(t) = \text{Law}(Y_u^{*,m}(t)) \mu(du)
\]
for every $t \in [0,T]$.
\end{itemize}

To solve step (ii) above, which is a standard optimal control problem for a finite state Markov chain, one introduces the value function

\begin{equation}\label{Vm}
\begin{split}
V^m(t,x,u)  =  \E_{Y^{\a^{*,m}}_u(t) = x} &\left[ \int_t^T L(Y^{\a^{*,m}}_u(s), \a^{*,m}(s, Y^{\a^{*,m}}_u(s)), u) ds \right.+ \\ 
 & \left.  \int_t^T F(Y^{\a^{*,m}}_u(s), m(s),u) ds + G(Y^{\a^{*,m}}_u(T), m(T),u) \right]
\end{split}
\end{equation}
Note that, being the value function of an optimal control problem,
\begin{equation} \label{infV}
\begin{split}
V^m(t,x,u)  =  \inf_{\a} \E_{Y^{\a}_u(t) = x} &\left[ \int_t^T L(Y^{\a}_u(s), \a(s, Y^{\a}_u(s)), u) ds \right.+ \\ 
 & \left. 
 \int_t^T F(Y^{\a}_u(s), m(s),u) ds + G(Y^{\a}_u(T), m(T),u) \right],
 \end{split}
\end{equation}
where the $\inf$ is over all locally integrable feedback controls.
For any fixed $u \in \mathcal{U}$, $V^m(t,x,u)$ satisfies the $|\S|$-dimensional ODE
\be \label{HJB}
\begin{split}
- \frac{d}{dt} V^m(t,x,u) + H(x, \nabla V^m(t,x,u),u) & = F(x,m(t),u) \\
V^m(T,x,u) & = G(x,m(T),u)
\end{split}
\ee
where $H: \S \times \R^{\S} \times \mathcal{U} \ra \R$ is defined by
\be \label{ham}
H(x,p,u) = - \inf_{a \in [0,+\infty)^{\S}} \left[ \sum_{y \neq x} a_y p_y + L(x,a,u) \right]
\ee
and, whenever $f: \S \ra \R$, we set $\nabla f(x) = (f(y) - f(x))_{y \in \S} \in \R^{\S}$. We will make assumptions on the function $L$ guaranteeing the existence of a unique minimizer $a^*(x,p,u)$ in \eqref{ham}. This implies that the optimal control is given by
\be \label{optcont}
\a^{*,m}(t,x,u) = a^*(x,\nabla V^m(t,x,u),u).
\ee
Thus, step (iii) above is equivalent to the fact that the flow $(m(t))_{t \in [0,T]}$ solves the forward Kolmogorov equation for the Markov chain with rates $\a^{*,m}$: if $m(t,u)$ denotes a regular version of the conditional distribution of $m(t)$, i.e. $m(t)(x,du) = m_x(t,u) \mu(du)$, we require
\be \label{Kfor}
\frac{d}{dt} m_x(t,u) = \sum_{y \in \S} m_y(t,u) a_x^*(y,\nabla V^m(t,y,u),u).
\ee
Putting together  \eqref{HJB} and \eqref{Kfor} we obtain the full system of equations:
\be \label{MFG}
\begin{split}
- \frac{d}{dt} V^m(t,x,u) & = - H(x, \nabla V^m(t,x,u),u) +  F(x,m(t),u) \\
V^m(T,x,u) & = G(x,m(T),u) \\
\frac{d}{dt} m_x(t,u) & = \sum_{y \in \S} m_y(t,u) a_x^*(y,\nabla V^m(t,y,u),u) \\
m_x(0,u) & = m_{x,0}(u)
\end{split}
\ee
Note that these equations are {\em coupled}, as $F$ and $G$ depend on the whole measure $m$.
In order to adapt to this spatially dependent context the standard results in Mean Field games we make the following assumptions. They could be weakened, but allow to illustrate the whole structure without technicalities.

\medskip

\noindent

{\em Assumption 1}. The function $L: \S \times [0,+\infty)^{\S} \times \mathcal{U} \ra \R$ is differentiable in the second variable, and we denote by $\nabla_a L(x,a,u)$ its differential;  both $L(x,a,u)$ and $\nabla_a L(x,a,u)$  are assumed to be continuous in $u$. Moreover, we assume there exists $\g>0$ such that for every $x \in \S$, $a,a' \in \R^{\S}$, and $u \in \mathcal{U}$
\be \label{convex}
L(x,a',u) \geq L(x,a,u) + \nabla_a L(x,a,u) \cdot (a' - a) + \g \|a'-a\|^2,
\ee
where ``$\cdot$'' denotes the standard scalar product on $\R^{\S}$. 
\medskip

\noindent
{\em Assumption 2}.
The functions $F,G: \S \times P(\mathcal{U} \times \S) \times \mathcal{U} \ra \R$ are bounded and 
 continuous on $\mathcal{P}_{\mu}\times  \mathcal{U}$.
 \medskip 

\noindent
{\em Assumption 3}. The functions $F,G: \S \times P(\mathcal{U} \times \S) \times \mathcal{U} \ra \R$ satisfy the following assumption: for every $m, \tilde{m} \in \mathcal{P}_{\mu}$

\be
\label{mon}
\sum_{x \in \S} \int \left[ F(x,m,u) - F(x,\tilde{m},u) \right] \left( m_x(u) - \tilde{m}_x(u) \right) \mu(du)  \geq 0.
\ee

\begin{thm} \label{th:existence}
Under Assumption 1 and 2, the Mean Field Game equations \eqref{MFG} admit a solution.
\end{thm}

\begin{thm} \label{th:uniqueness}
Under Assumption 1, 2 and 3, the Mean Field Game equations \eqref{MFG} admit a unique solution.
\end{thm}

\section{The Approximation Theorem}

To prove the approximation result we need another regularity assumption on the functions $F$ and $G$ and we have to 
 specify the mode of convergence of $F_N, G_N$ to $F,G$.
\medskip

\noindent
{\em Assumption 4}. The function $F$ and $G$ are such that 
 if a sequence $m_n\in P(\mathcal{U} \times \S)$ converges weakly to $m\in \mathcal{P}_{\mu}$ then $F(x,m_n,u) \to F(x,m,u)$ uniformly in $u\in\mathcal{U}  $, same for $G$.

\medskip

\noindent
{\em Assumption 5(a)}. For $Q=F$ and $Q=G$ we have
\[
\lim_{N \ra +\infty} \sup_{{\bm x} \in \S^N, 1\leq i\leq N} \left| Q_N(x_i, m^{N,i}_{\bm{x}},u_i^N) - Q(x_i, m^{N,i}_{\bm{x}},u_i^N)\right| = 0.
\]

\medskip
\noindent
{\em Assumption 5(b)}. For $Q=F$ and $Q=G$ we have
\[
\lim_{N \ra +\infty}\sup_{{\bm x} \in \S^N} \frac{1}{N}\sum_{i=1}^N  \left| Q_N(x_i, m^{N,i}_{\bm{x}},u_i^N) - Q(x_i, m^{N,i}_{\bm{x}},u_i^N)\right| = 0.
\]
Note that Assumption 5(a) is stronger than assumption 5(b).
\begin{thm} \label{th:approx}
Suppose Assumption 1, 2, 4 and 5(b) hold. Let $(V,m)$ be a solution of the Mean Field Game Equation \eqref{MFG}, and set
\[
\a^{*}(t,x,u) = a^*(x,\nabla V(t,x,u),u),
\]
where $a^*(x,p,u)$ is the unique minimizer  in \eqref{ham}. Denote by ${\bs \a}^N$ the feedback control for the $n$-player game given by
\[
\a^{N,i}(t, {\bs x}) = \a^{*}(t,x_i,u_i^N),
\]
Then there exist sequences $\varepsilon_N, \delta_N \ra 0$ as $N \ra +\infty$ such that ${\bs \a}^N$ is a $\varepsilon_N$-$\delta_N$-Nash equilibrium for the $N$-player game. In the case Assumption 5(a) also hold, then we can take $\delta_N \equiv 0$.
\end{thm}

\begin{rem} \label{rem:alfa}
In Theorem \ref{th:approx} the assumption that the players' strategies $\a$ coincide with the jump intensity is quite relevant. More generally, one could consider models in which the jump intensity is a function of the strategy and of the empirical measure. This on one hand brakes the convexity of the model, which could be solved by considering relaxed strategies (see \cite{CF20}). On the other hand, the states of the $N$ players using the strategy $\a^*$ would not be anymore independent. The approximation theorem would require a {\em propagation of chaos} result for these extended mean field models; this is, we believe, a problem of independent interest.
\end{rem}

\begin{rem}
\label{rem:quant}
It would be desirable to prove a {\em quantitative} approximation theorem, i.e. with explicit estimates for $\varepsilon_N$ and $\delta_N$. This would require more explicit assumptions on the convergence of $F_N$ and $G_N$ to their limits. We have decided to postpone this study to a forthcoming paper, which will include another related topic, namely the approximation property of the feedback produced by the {\em Master Equation} (see \cite{CP19}) rather than by Equation \eqref{MFG}. This feedback is expected to have a better behavior, and in this context quantitative estimates are crucial.
\end{rem}

\section{Examples}

Is this section we illustrate some classes of examples for the functions $F$ and $G$, and discuss different $N$-player games which are approximated by the same Long Range Game.

\subsection{Models with two body interaction}

In this section we present some models where  the two functions $F$ and $G$ have the form given by (\ref{effe}) below. In what follows we write all the statements and assumptions only for the function $F$, implying the same also for $G$. 

Assume the function $F$ (and $G$) are of the type:
\be\label{effe}
F(x,m,u):=\int_{\mathcal{U}\times \S}K(u,v)f(x,y)m(dv,dy)
\ee
and  where 
\begin{itemize}
\item[(a)] $f:\S\times\S \to R$ and $K:\mathcal{U}\times \mathcal{U}\to \R$ are  bounded functions, and
\item[(b)]  $\forall \, \epsilon >0$ there exists an open set $A_{\epsilon}\subseteq \mathcal{U}\times \mathcal{U}$ such that: 
$K(\cdot,\cdot)$ is uniformly continuous in $A_{\epsilon}$, and, denoting by $B^u_\epsilon=\{ v\, :\, (u,v)\in A_\epsilon\}$, 
 we have $\mu\left(\left( B^u_\epsilon\right)^c\right) \leq \epsilon$ $\forall u\in \mathcal{U}$.
\end{itemize}
We remark that, with minor modifications, models with {\em $k$-body interactions} could be dealt with as well, i.e. models for which $F$ (and $G$) have the form
\[
F(x,m,u):=\int_{\mathcal{U}^{k-1}\times \S^{k-1}} K(u,v_1,\ldots,v_{k-1}) f(x,y_1,\ldots,y_{k-1}) m(dv_1,dy_1) \cdots m(dv_{k-1},dy_{k-1}).
\]
We omit the detains of this extension.

\bigskip

Next lemma proves  that if $F$ (and so also $G$) is of the form \eqref{effe}, then Assumption 2 is satisfied.

\begin{lem}\label{continuity}
Under assumptions (a) and (b) above, $F$ in (\ref{effe}) is continuous on $\mathcal{P}_\mu \times \mathcal{U}$.
\end{lem}
\begin{proof}
Let $u_n$ be a sequence in $\mathcal{U}$, converging to $u\in \mathcal{U}$ and $m_n\in \mathcal{P}_\mu$ converging weakly to $m\in \mathcal{P}_\mu$. Then
\begin{equation}{\label{first-second}}
\begin{split}
\left| F(x,m_n,u_n)-F(x,m,u)\right|\leq\int_{\mathcal{U}\times \S}\left| K(u_n,v)-K(u,v)\right| |f(x,y)|m_n(dv,dy) & \\
+\left| \int_{\mathcal{U}\times \S} K(u,v) f(x,y)m_n(dv,dy)- \int_{\mathcal{U}\times \S} K(u,v) f(x,y)m(dv,dy)\right|.
\end{split}
\end{equation}
First we deal with the second term in the r.h.s. of the previous equation. Assumption (b) implies, in particular, that for each $u \in \mathcal{U}$
\[
\mu\left(\{v\in \mathcal{U}\, |\, v \mapsto K(u,v) \text{ is continuous in $v$ }\}\right)=1.
\]
Since both $K$ and $f$ are bounded, $\S$ is a finite set, and $m_n$ converges weakly to $m$, by the Portmanteau Theorem (see \cite{Le12}, Section 8.5), the second term goes to zero.

Now we look at the first term of (\ref{first-second}).
Denote by $C>0$ the upper bound for $f$. For any $\epsilon >0$, it holds
\[
\int_{\mathcal{U}\times \S}\left| K(u_n,v)-K(u,v)\right| |f(x,y)|m_n(dv,dy)\leq C\int_{\mathcal{U}}\left| K(u_n,v)-K(u,v)\right| \mu(dv)
\]
\[
\leq  C\int_{B^u_\epsilon}\left| K(u_n,v)-K(u,v)\right| \mu(dv)+
C\int_{\left(B^u_\epsilon\right)^c}\left| K(u_n,v)-K(u,v)\right| \mu(dv)
\]
On $B^u_\epsilon$, $K$ is uniformly continuous, thus the first term converge to zero, by the dominated convergence's theorem. On the other hand, since $K$ is bounded,
$\left| K(u_n,v)-K(u,v)\right|\leq \hat{C}$ for some $\hat{C}>0$, so:
\[
\int_{\left(B^u_\epsilon\right)^c}\left| K(u_n,v)-K(u,v)\right| dv\leq \hat{C}\mu\left( \left(B^u_\epsilon\right)^c\right)\leq \hat{C}\epsilon.
\]
By the arbitrariness of $\epsilon>0$,  we have that
\[
\int_{\mathcal{U}\times \S}\left| K(u_n,v)-K(u,v)\right| |f(x,y)|m_n(dv,dy)
\]
goes to zero as $n \ra +\infty$; 
continuity of $F$ on  $\mathcal{P}_\mu\times {\mathcal{U}}$ thus follows.

\end{proof}

Now we prove  that if $F$ (and so also $G$) is of the previous form (see (\ref{effe})), then also Assumption 4 is satisfied.

\begin{lem}\label{uniforme}
Let $F$ be as in (\ref{effe}), and assume assumptions (a) and (b) hold. Then 
for any  sequence $m_n\in P({\mathcal{U}} \times \S)$ converging  weakly to $m\in \mathcal{P}_{\mu}$ we have  $F(x,m_n,u) \to F(x,m,u)$ uniformly in $u\in{\mathcal{U}}$
\end{lem}
\begin{proof}
Fix $u\in {\mathcal{U}}$, and $\epsilon>0$.

Since both $K$ and $f$ are bounded, $\S$ is a finite set, $\mu\left(\{v\in \mathcal{U}\, |\, K(u,v) \text{ is continuous }\}\right)=1$, and $m_n$ converges weakly to $m$,  by the Portmanteau Theorem (see \cite{Le12}, Section 8.5), $F(x,m_n,u) \to F(x,m,u)$ {\em pointwise} in $u$. We need to show that this convergence is actually uniform in $u$. By pointwise convergence
we  have  that for each $u \in \mathcal{U}$
 there exists $N^1_{\epsilon,u}\in \N$ such that 
 \be\label{primaunif}
\left| F(x,m_n,u)-F(x,m,u)\right|\leq \epsilon
\ee
 for all $n>N^1_{\epsilon,u}$.
 
Moreover, for any other  $u'\in {\mathcal{U}}$, we have:
\be\label{effeu'}
\left| F(x,m_n,u')-F(x,m,u')\right|\leq 
\ee
\[
\left| F(x,m_n,u')-F(x,m_n,u)\right| +\left| F(x,m_n,u)-F(x,m,u)\right|+\left| F(x,m,u)-F(x,m,u')\right|
\]
Note that
\[
\left| F(x,m_n,u')-F(x,m_n,u)\right| \leq \int_{{\mathcal{U}}\times \S}\left| K(u',v)-K(u,v)\right| |f(x,y)| m_n(dv,dy).
\]
Denoting by $C>0$ an upper bound for $f$, and by $\mu_n$ the marginal of $m_n$ on $\mathcal{U}$, we have
\[
\left| F(x,m_n,u')-F(x,m_n,u)\right| \leq C \int_{{\mathcal{U}}} \left| K(u',v)-K(u,v)\right| \mu_n(dv)
\]
\[
=
C \int_{B^u_\epsilon} \left| K(u',v)-K(u,v)\right| \mu_n(dv)+  C \int_{\left(B^u_\epsilon\right)^c} \left| K(u',v)-K(u,v)\right| \mu_n(dv)
\]
Since $K$ is uniformly continuous in $B^u_\epsilon$, there exists an open neighborhood $\mathcal{O}^1_u\subseteq \mathcal{U}$ of $u$  such that  if $u'\in \mathcal{O}^1_u$ then $ \left| K(u',v)-K(u,v)\right| \leq \epsilon$, which implies 
\be\label{suB}
 \int_{B^u_\epsilon} \left| K(u',v)-K(u,v)\right| \mu_n(dv)\leq \epsilon.
\ee
Moreover by boundedness of $K$, we have $\left| K(u_n,v)-K(u,v)\right|\leq \hat{C}$ for some $\hat{C}>0$, so
\[
 \int_{\left(B^u_\epsilon\right)^c} \left| K(u',v)-K(u,v)\right| \mu_n(dv)\leq \hat{C} \mu_n\left(\left(B^u_\epsilon\right)^c\right).
\]
Since $\mu_n$ converges to $m$, $\left(B^u_\epsilon\right)^c$ is a closed set, and $\mu \left(\left(B^u_\epsilon\right)^c\right) \leq \epsilon$, by the Portmanteau Theorem (see \cite{Le12}, Section 8.5), there exists $N_{\epsilon,u}^2$ such that $m_n\left({\left(B^u_\epsilon\right)^c}\right)\leq 2\epsilon$ for all $n\geq N_{\epsilon,u}^2$, so
\be\label{suBc}
\int_{\left(B^u_\epsilon\right)^c} \left| K(u',v)-K(u,v)\right| \mu_n(dv)\leq 2 \hat{C} \epsilon
\ee
Thus, for any $n\geq N_{\epsilon,u}^2$ and any $u'\in \mathcal{O}^1_u$, from equations (\ref{suB}) and (\ref{suBc}) we have:
\begin{equation}\label{new}
\left| F(x,m_n,u')-F(x,m_n,u)\right| \leq \left(C+2 C\hat{C} \right) \epsilon.
\end{equation}

By Lemma \ref{continuity}, we know that $F$ is continuous on $\mathcal{P}_\mu\times \mathcal{U}$, thus there exists another neighborhood $\mathcal{O}^2_u\subseteq \mathcal{U}$ uf $u$ such that if $u'\in \mathcal{O}^2_u$ then 
\be{\label{ultima}}
\left| F(x,m,u)-F(x,m,u')\right|\leq \epsilon,
\ee
 since $m\in \mathcal{P}_\mu$.

Let $\mathcal{O}_u=\mathcal{O}_u^1\cap \mathcal{O}_u^2$ and $N_{\epsilon,u}=\max\{N^1_{\epsilon,u},\, N_{\epsilon,u}^2\}$. Using equations (\ref{primaunif}), (\ref{new}), and (\ref{ultima}), 
we have that if $u'\in \mathcal{O}_u$ and $n>N_{\epsilon,u}$, 
\[
\left| F(x,m_n,u')-F(x,m,u')\right|\leq \left(C\epsilon+ 2C\hat{C}\right)\epsilon+\epsilon+\epsilon=\left(C+2C\hat{C}+2\right)\epsilon
\]
Now we conclude using the compactness of $\mathcal{U}$. In fact the open neighborhoods $(\mathcal{O}_u)_{u \in \mathcal{U}}$ gives an open cover of $\mathcal{U}$, so there exists a finite cover, say $\mathcal{O}_{u_1},\cdots, \mathcal{O}_{u_r}$, for some $r>0$. Let 
$N_\epsilon=\max_{i=1}^r N_{\epsilon,u_i}$, we have that for all $n>N_\epsilon$ and for all $u\in\mathcal{U}$
\[
\left| F(x,m_n,u)-F(x,m,u)\right|\leq \left(C+2C\hat{C}+2\right)\epsilon.
\]
\end{proof}
There are several $N-$player games possessing the same limit game where $F$ and $G$ are of the type \eqref{effe}.

\noindent
{\bf{Example 1}}
 Let $(U_i)_{i \geq 1}$ be a sequence of $\mathcal{U}$-valued, i.i.d. random variables with law $\mu$. Set
\[
u_i^N := U_i.
\]
Moreover choose $F_N = F$, $G_N = G$, i.e.
\[
F_N\big(X_i(t), m^{N,i}_{\bm{X}}(t), u_i^N \big) = \frac{1}{N-1} \sum_{j : j \neq i} K(u_i^N,u_j^N) f(x_i,x_j).
\]
For applications of these models see \cite{ACDL22}.
Assumption 5(a) holds trivially, thus, by theorem \ref{th:approx}, equilibria for the limit game are $\e_N$-Nash for the $N$-player game.

\noindent
{\bf{Example 2}}
Let $\mathcal{U} = [0,1]$, $\mu = \lambda $ Lebesgue measure on $[0,1]$,   $u_i^N := \frac{i}{N}$ and set
\begin{equation}{\label{effeN}}
F_N\big(X_i(t), m^{N,i}_{\bm{X}}(t), u_i^N \big) = \frac{1}{N-1} \sum_{j : j \neq i} K\left(\frac{i}{N},\frac{j}{N}\right) f(x_i,x_j),
\end{equation}
which gives the  Kac-type interaction.  Also in this case, since $F_N = F$ and $G_N = G$, assumption 5(a) holds trivially, thus equilibria for the limit game are $\e_N$-Nash for the $N$-player game.

\noindent
{\bf{Example 3}}
A relevant generalization  of the previous example is obtained by assuming that, for the $N-$player models the functions $F_N$, and $G_N$ are as in equation \eqref{effeN}, but now instead of the function $K$ we have, for any $N$, different  functions $K_N$,  which are defined on the set $\left\{\left(\frac{i}{N},\frac{j}{N}\right): \, i,j=1,\ldots,N \right\}$, and we assume that  
 $K^N \rightarrow K$ in the {\em{ cut norm}}, i.e.
\[
\|K^N - K \|_{\square}   := \sup_{S,T \in \mathcal{B}([0,1])}\left|\int_{S \times T} [K^N(u,v) - K(u,v)] du dv \right| \rightarrow 0
\]
where $K^N$ is meant to be extended to $[0,1)^2$ as the step function
\[
K^N(u,v) := K^N \left( \frac{\lfloor Nu +1 \rfloor}{N},  \frac{\lfloor Nv +1 \rfloor}{N} \right).
\]
It has been proved (see \cite{LO12}, Lemma 8.11) that the cut norm is equivalent to the following {\footnotesize{ $\infty \to 1$}}-norm: for any $H:[0,1]^2\to \R$  
\[
\| H\|_{\infty\to 1}:= \sup_{||\varphi||_{\infty}\leq 1}\int_0^1\left|\int_0^1 H(u,v)\varphi(v) dv\right|du
\]

The following fact holds:
\begin{prop}\label{cut-norm}
Assume that $\|K^N - K \|_{\square} \to 0$ and that $K$ satisfies assumption (a) and (b), then the   functions $F_N$ and $F$ (and $G_N$ and $G$) satisfy assumption 5(b).
\end{prop}
\begin{proof}
We may assume, without loss of generality, that:
$f(x,y)= h(x)k(y).$
In fact, since $\Sigma$ is a finite set, $f$ can be written as $$f(x,y)=\sum_{l=1}^m f_l(x,y)= \sum_{l=1}^m h_l(x)k_l(y),$$ thus $F_N=\sum_{l=1}^mF^l_N$, similarly for $F$, and the convergence of all summands imply the convergence of $F_N$.

We have:
\begin{equation}\label{uno}
\begin{split}
\left|F_N\left(x_i, m^{N,i}_{\underline{x}}, \frac{i}{N} \right) \right. -  & \left.
F\left(x_i, m^{N,i}_{\underline{x}}, \frac{i}{N}\right) \right| \\
& = \left|
\frac{1}{N-1}\sum_{\stackrel{j=1}{{\mbox{\tiny $ j\neq i$}}}}^N\left( K^N\left(\frac{i}{N},\frac{j}{N}\right) -
K\left(\frac{i}{N},\frac{j}{N} \right)\right) h(x_i)k(x_j)\right|  \\
& \leq ||h||_{\infty}||k||_{\infty}\left|
\frac{1}{N-1}\sum_{\stackrel{j=1}{{\mbox{\tiny $ j\neq i$}}}}^N\left( K^N\left(\frac{i}{N},\frac{j}{N}\right) -
K\left(\frac{i}{N},\frac{j}{N} \right)\right) \varphi_N\left(\frac{j}{N}\right) \right|  
\end{split}
\end{equation}
where $\varphi_N\left(\frac{j}{N}\right):=\frac{k(x_j)}{||k||_{\infty}}$.
Denote by:
\[
\Delta_N(i)=\left|
\frac{1}{N-1}\sum_{\stackrel{j=1}{{\mbox{\tiny $ j\neq i$}}}}^N\left( K^N\left(\frac{i}{N},\frac{j}{N}\right) -
K\left(\frac{i}{N},\frac{j}{N} \right)\right) \varphi_N\left(\frac{j}{N}\right) \right|  
\]
Define:
\[
\hat{\Delta}_N(i):=\left|
\frac{1}{N}\sum_{j=1}^N\left( K^N\left(\frac{i}{N},\frac{j}{N}\right) -
K\left(\frac{i}{N},\frac{j}{N} \right)\right) \varphi_N\left(\frac{j}{N}\right) \right|  
\]
Then, letting for any $v\in [0,1]$, let $\varphi_N(v)=\varphi_N\left(\frac{\lfloor Nv\lfloor}{N}\right)$, and denoting   by $\tilde{K}^N(u,v)=K\left(\frac{\lfloor Nu\lfloor}{N},\frac{\lfloor Nv\lfloor}{N}\right)$, we have:
\[
\hat{\Delta}_N(i)=\left|\int_0^1\left( K^N\left(\frac{i}{N},v\right)-\tilde{K}^N\left(\frac{i}{N},v\right)\right)\varphi_N(v) dv \right|
\]
Notice also that:
\begin{multline*}
\left| \Delta_N(i)-\hat{\Delta}_N(i)\right| \\ \leq
\left|
\sum_{\stackrel{j=1}{{\mbox{\tiny $ j\neq i$}}}}^N\left( K^N\left(\frac{i}{N},\frac{j}{N}\right) -
K\left(\frac{i}{N},\frac{j}{N} \right)\right) \varphi_N\left(\frac{j}{N}\right)\left(\frac{1}{N-1}-\frac{1}{N}\right)\right|
\\ +\left| \frac{1}{N}\left( K^N\left(\frac{i}{N},\frac{i}{N}\right) -
K\left(\frac{i}{N},\frac{i}{N} \right)\right)\varphi_N\left(\frac{i}{N}\right)\right|
\end{multline*}
By assumption $K$ and $K^N$ are bounded by a common bound $C>0$. From the previous inequality we have:
\[
\left| \Delta_N(i)-\hat{\Delta}_N(i)\right|  \leq \frac{2C}{N}.
\]
Thus we conclude:
\[
 \Delta_N(i)\leq \left|\int_0^1\left( K^N\left(\frac{i}{N},v\right)-\tilde{K}^N\left(\frac{i}{N},v\right)\right)\varphi_N(v) dv \right| + \frac{2C}{N}.
 \]
From equation \eqref{uno} and from  the previous inequality, we have:
\[
\frac{1}{N} \sum_{i=1}^N \left|F_N\left(x_i, m^{N,i}_{\underline{x}}, \frac{i}{N} \right) 
F\left(x_i, m^{N,i}_{\underline{x}}, \frac{i}{N}\right) \right| 
\]
\[
\leq  ||h||_{\infty}||k||_{\infty}\left( \sum_{i=1}^N  \left | \int_0^1 \left( K^N\left(\frac{i}{N},v\right)-\tilde{K}^N\left(\frac{i}{N},v\right)\right)\varphi_N(v)dv \right|  +\frac{C}{N}\right)
\]
\[
= ||h||_{\infty}||k||_{\infty}\left( \int_0^1\left | \int_0^1 \left( K^N\left(\frac{i}{N},v\right)-\tilde{K}^N\left(\frac{i}{N},v\right)\right)\varphi_N(v)dv \right|  +\frac{C}{N}\right) 
 \]
 \[
\leq  ||h||_{\infty}||k||_{\infty}\left( \sup_{\| \phi\|\leq 1} \int_0^1\left | \int_0^1 \left( K^N\left(\frac{i}{N},v\right)-\tilde{K}^N\left(\frac{i}{N},v\right)\right)\phi(v))dv \right|  +\frac{C}{N}\right).
  \]
 Thus it holds:
 \begin{equation}\label{dis}
 \sup_{x\in \Sigma^N}\frac{1}{N} \sum_{i=1}^N \left|F_N\left(x_i, m^{N,i}_{\underline{x}}, \frac{i}{N} \right) 
F\left(x_i, m^{N,i}_{\underline{x}}, \frac{i}{N}\right) \right| \leq 
 ||h||_{\infty}||k||_{\infty} \left(\| K^N-\tilde{K}^N\|_{\infty\to1}+\frac{C}{N}\right).
 \end{equation}
Since, by assumption $\|K^N - K \|_{\square} \to 0$, which is equivalent to $\|K^N - K \|_{\infty \to 1} \to 0$, to conclude that assumption 5(b) holds, i.e. 
\[
\lim_{N\to+\infty} \sup_{x\in \Sigma^N}\frac{1}{N} \sum_{i=1}^N \left|F_N\left(x_i, m^{N,i}_{\underline{x}}, \frac{i}{N} \right) 
F\left(x_i, m^{N,i}_{\underline{x}}, \frac{i}{N}\right) \right|=0
\]
by equation \eqref{dis}, it is enough to prove that 
\begin{equation}\label{appr}
\|\tilde{K}^N - K \|_{\infty \to 1} \to 0
\end{equation}
We have:
\[
\|\tilde{K}^N - K \|_{\infty \to 1} =\sup_{\|\phi\|\leq 1}\int_0^1\left |\int_0^1 \left(\tilde{K}^N(u,v)-K(u,v)\right) \phi(v) dv \right| du 
\]
\[
\leq \int_0^1\int_0^1 \left | \tilde{K}^N(u,v)-K(u,v)\right|  dv  du.
\]
For each $\epsilon >0$, by assumption (b) on the function $K$, we know that there exists 
an open set $A_{\epsilon}\subseteq [0,1]^2$ such that: 
$K(\cdot,\cdot)$ is uniformly continuous in $A_{\epsilon}$, so we have:
\[
\|\tilde{K}^N - K \|_{\infty \to 1} \leq 
\underbrace{\iint_{A_{\epsilon}}\left | \tilde{K}^N(u,v)-K(u,v)\right|  dv  du}_{I_N^1}+ \underbrace{\iint_{(A_{\epsilon})^c}\left | \tilde{K}^N(u,v)-K(u,v)\right|  dv  du}_{I_N^2}
\]
Since on $A_{\epsilon}$ the function $K$ is uniformly continuous, then $\tilde{K}^N(u,v)\to K(u,v)$ as $N\to +\infty$, so $\lim_{N\to +\infty} I_N^1 = 0$, by dominated convergence theorem.
On the other hand, we have:
\[
I_N^2\leq 2\|K\|_{\infty}\int_0^1\lambda\left(\left(B^u_{\epsilon}\right)^c\right) du\leq  2\|K\|_{\infty} \epsilon.
\]
So equation \eqref{appr} holds, and the statement is proved.
\end{proof}

In this case Assumption 5(b) holds, i.e. equilibria for the limit game are $\e_N$-$\delta_N$-Nash for the $N$-player game.

One example of a sequence $K^N$ such that $\|K^N - K \|_{\square} \to 0$ is given by the random adjacency matrices whose entries are independent random variables
\[
K^N\left(\frac{i}{N},\frac{j}{N} \right) \sim Be\left(K\left(\frac{i}{N},\frac{j}{N} \right)\right),
\]
where we assume $K:[0,1] \times [0,1] \ra [0,1]$.
A proof of this fact is not so easy to find in the literature, so we sketch it here for completeness.
\begin{prop} \label{prop:erdos}
Suppose $K:[0,1] \times [0,1] \ra [0,1]$ satisfies assumptions (a) and (b), and let 
\[
K^N\left(\frac{i}{N},\frac{j}{N} \right) \sim Be\left(K\left(\frac{i}{N},\frac{j}{N} \right)\right)
\]
be independent random variables. Then $\|K^N - K \|_{\square} \to 0$ almost surely.
\end{prop}
\begin{proof}
Let $\tilde{K}^N(u,v)=K\left(\frac{\lfloor Nu\lfloor}{N},\frac{\lfloor Nv\lfloor}{N}\right)$. In the proof of Proposition \ref{cut-norm} we have seen that $\|\tilde{K}^N - K \|_{\square} \to 0$. Thus, is suffices to show that $\|K^N - \tilde{K}^N \|_{\square} \to 0$ almost surely. Noting that $K^N$ and $\tilde{K}^N$ are constant on the squares $\Big(\frac{i-1}{N}, \frac{i}{N} \Big] \times \Big(\frac{j-1}{N}, \frac{j}{N} \Big]$, we have that for each $S,T \in \mathcal{B}(0,1)$
\[
\int_{S \times T} \left[K^N(u,v) - \tilde{K}^N(u,v) \right] du dv = \frac{1}{N^2} \sum_{i,j=1}^N a_{i} b_j \left[K^N\left(\frac{i}{N},\frac{j}{N}\right) - K\left(\frac{i}{N},\frac{j}{N}\right) \right],
\]
where 
\[
\begin{split}
a_{i} & = N \lambda\left(S \cap \Big(\frac{i-1}{N}, \frac{i}{N} \Big]\right)  \in [0,1] \\
b_j & = N \lambda\left(T \cap \Big(\frac{j-1}{N}, \frac{j}{N} \Big] \right) \in [0,1]
\end{split}
\]
This implies that 
\[
\|K^N - \tilde{K}^N \|_{\square} =  \sup_{a_i, b_j \in [0,1]} \frac{1}{N^2} \left| \sum_{i,j=1}^N a_{i}b_j \left[K^N\left(\frac{i}{N},\frac{j}{N}\right) - K\left(\frac{i}{N},\frac{j}{N}\right) \right] \right|.
\]
The function $(a_i, b_j)_{i,j=1}^N  \mapsto \left| \sum_{i,j=1}^N a_i b_j \left[K^N\left(\frac{i}{N},\frac{j}{N}\right) - K\left(\frac{i}{N},\frac{j}{N}\right) \right] \right|^2$ is easily checked to be convex in $(a_i)_{i=1}^N$ for fixed $(b_j)_{j=1}^N$, and convex in $(b_j)_{j=1}^N$ for fixed $(a_i)_{i=1}^N$. This is enough to conclude that
its supremum on $[0,1]^{2N}$ is attained on extremals, so
\[
\|K^N - \tilde{K}^N \|_{\square} = \sup_{a_i, b_j \in \{0,1\}\}} \frac{1}{N^2} \left| \sum_{i,j=1}^N a_i b_j \left[K^N\left(\frac{i}{N},\frac{j}{N}\right) - K\left(\frac{i}{N},\frac{j}{N}\right) \right] \right|.
\]
Note that $X_{ij} := a_i b_j \left[K^N\left(\frac{i}{N},\frac{j}{N}\right) - K\left(\frac{i}{N},\frac{j}{N}\right) \right]$ are bounded, independent, mean zero random variables. By Hoeffding's inequality (\cite{Hof63}) there is a constant $C>0$ such that
\[
\P\left(\left| \sum_{i,j=1}^N X_{ij} \right| > t \right) \leq e^{-\frac{t^2}{CN^2}}.
\]
Choosing $t = N^{7/4}$ we have, for fixed $(a_1,\ldots,a_n,b_1,\ldots b_N) \in \{0,1\}^{2N}$
\[
\P\left(\frac{1}{N^2} \left| \sum_{i,j=1}^N a_i b_j \left[K^N\left(\frac{i}{N},\frac{j}{N}\right) - K\left(\frac{i}{N},\frac{j}{N}\right) \right] \right| > N^{-1/4} \right) \leq e^{-\frac{N^{3/2}}{C}}.
\]
Taking the supremum over $a_i,b_j$
\[
\P\left(\|K^N - \tilde{K}^N \|_{\square} > N^{-1/4} \right) \leq 2^{2N} e^{-\frac{N^{3/2}}{C}}.
\]
Since this last sequence is summable in $N$, we can apply Borel-Cantelli lemma, which states that, with probability one, $\|K^N - \tilde{K}^N \|_{\square} \leq N^{-1/4}$ except for finitely many $N$, and the conclusion follows.

\end{proof}

\subsection{Low resolution models}

We consider here models in which the player at a given position $u \in \mathcal{U}$ does not directly ``observe'' the state of players at any other positions $v$, but rather a {\em local average} of states of players around position $v$. More specifically, suppose we are given functions $f:\S \times P(\S) \ra \R$, bounded and continuous in the second variable, $K:\mathcal{U} \times \mathcal{U} \ra \R$ continuous (but a weaker continuity as in assumption (b) in the previous section could be dealt with as well), $\varphi: \mathcal{U} \times \mathcal{U} \ra (0,+\infty)$ continuous, and assume $F$ (and $G$) are of the form, for $m(y,dv) = m_y(v) m^{(2)}(dv)$:
\be \label{lowres}
F(x,m,u) = \int K(u,v) f(x,m^{\varphi}(v)) m^{(2)}(dv)
\ee
where
\[
m^{\varphi}(v) = \left(m^{\varphi}_y(v)\right)_{y \in \S},
\]
and
\[
m^{\varphi}_y(v) := \frac{\int \varphi(v,w) m_y(w) m^{(2)}(dw)}{\int \varphi(v,w) m^{(2)}(dw)}.
\]
It is easily checked that Assumption 2 and Assumptions 4 hold.

Note that, assuming $F^N = F$, for the $N$-player game the cost has the form
\[
F(x_i, m^{N,i}_{{\bf x}}, u_i^N) = \frac{1}{N-1} \sum_{j : j \neq i} K(u_i^N, u_j^N)f\left(x_i , \left(\frac{ \sum_{k:k \neq j} \varphi(u_j,u_k) {\bf 1}_{\{x_k = y\}}}{ \sum_{k:k \neq j} \varphi(u_j,u_k)}\right)_{y \in \S}\right)
\]
For the limit game, for which only $m \in \mathcal{P}_{\mu}$ matters, one could consider the limit case in which $\varphi(v,w) \mu(dw) \simeq \delta_v(dw)$, where $\delta_v$ denotes the Dirac measure. This would amount to, for $m \in \mathcal{P}_{\mu}$,
\[
F(x,m,u) =  \int K(u,v) f\left(x,(m_y(v))_{y \in \S} \right) \mu(dv).
\]
This case is singular, in the sense that in general this $F$ is not necessarily continuous in $\mathcal{P}_{\mu}$.

\section{Proofs}
\subsection{Proof of Theorem \ref{th:existence}}

We first prove two Lemmas, and then we give the proof of Theorem \ref{th:existence}.

\begin{lem}\label{V-continua}
The function  $V^m(t,x,u)$ is bounded, and it is continuous as a function of $m$ on $\mathcal{P}_{\mu}$.
\end{lem}
\begin{proof}
Taking a constant control $\bar{\a} \in [0,+\infty)^{\S}$, by \eqref{infV}, 
\begin{equation*}
\begin{split}
V^m(t,x,u) \leq \E_{Y^{\bar{\a}}_u(t) = x} &\left[ \int_t^T L(Y^{\bar{\a}}_u(s), \bar{\a}, u) ds \right.+ \\ 
 & \left. 
 \int_t^T F(Y^{\bar{\a}}_u(s), m(s),u) ds + G(Y^{\bar{\a}}_u(T), m(T),u) \right]
 \end{split}
 \end{equation*}
 which is bounded since $F$ and $G$ are bounded, and so it is $u \mapsto L(y, \bar{\a}, u)$, as it is continuous in a compact space. This shows that $V^m$ is bounded from above. Boundedness from below follows from \eqref{Vm}, boundedness of $F$ and $G$ and the fact that, by \eqref{convex}, $L(y, \bar{\a}, u)$ is bounded from below.

We now prove continuity in $m$. For any  $m_1, m_2\in \mathcal{P}_{\mu}$, by \eqref{HJB} we have:
\[
V^{m_1}(t,x,u)-V^{m_2}(t,x,u)= -\int_t^T \left( H(x, \nabla V^{m_1}(s,x,u),u)-  H(x, \nabla V^{m_2}(s,x,u),u)\right) ds+
\]
\[+ \int_t^T\left( F(x,m_1(s),u)-F(x,m_2(s),u)\right) ds+ G(x,m_1(T),u)-G(x,m_2(T),u)
\]
The function $H(x,p,u)$ is Lipschitz in its second variable, and denote by $k(u)$ its Lipschitz constant
(since $\S$ is finite we may assume that the Lipschitz constant depends only on $u$). Thus we have:
\[
|V^{m_1}(t,x,u)-V^{m_2}(t,x,u)|\leq \int_t^T  k(u)   \| \nabla V^{m_1}(s,x,u)- \nabla V^{m_2}(s,x,u) \| ds+
\]
\[+ \int_t^T\left| F(x,m_1(s),u)-F(x,m_2(s),u)\right| ds+ |G(x,m_1(T),u)-G(x,m_2(T),u)|
\]

Moreover we have:
\begin{equation*}
\begin{split}
  \| \nabla V^{m_1}(s,x,u)&- \nabla V^{m_2}(s,x,u) \| 
 \\ & \leq \sum_{y\in \S}|V^{m_1}(s,y,u)-V^{m_1}(s,x,u)
  -V^{m_2}(s,y,u)-V^{m_2}(s,x,u)|  \\
&
   \leq 2d\sup_{x\in \S} \left | V^{m_1}(s,x,u)-V^{m_2}(s,x,u) \right |
\end{split}
\end{equation*}
Letting
\[
M^u(t)=\sup_{x\in\S} \left | V^{m_1}(t,x,u)-V^{m_2}(t,x,u) \right |
\]
we have:
\[
|V^{m_1}(t,x,u)-V^{m_2}(t,x,u)|\leq \int_t^T k(u)2d M^u(s) ds+
\]
\[
+ \int_t^T\sup_{x\in \S}\left| F(x,m_1(s),u)-F(x,m_2(s),u)\right | ds+ \sup_{x\in \S}|G(x,m_1(T),u)-G(x,m_2(T),u)|
\]
Thus:
\[
M^u(t)\leq 
 2dk(u) \int_t^T M^u(s) ds + \]
\[+  \left( \int_t^T\sup_{x\in \S}\left| F(x,m_1(s),u)-F(x,m_2(s),u)\right | ds+ \sup_{x\in \S}|G(x,m_1(T),u)-G(x,m_2(T),u)|\right) 
 \]
 By applying the Bellman-Gronwall's Lemma, we get:
\begin{equation*}
\begin{split}
M^u(t)\leq & \left(  \int_t^T\sup_{x\in \S}\left| F(x,m_1(s),u)-F(x,m_2(s),u)\right | ds \right. \\ & \left.+ \sup_{x\in \S}|G(x,m_1(T),u)-G(x,m_2(T),u)| \right) e^{2dk(u)(T-t)}
\end{split}
\end{equation*}
Since the functions $G$ and $F$ are continuous on $ \mathcal{P}_{\mu}\times \mathcal{U}$, we have that also the function $V^m$ is continuous as a function on $m$, with $m$ varying in $ \mathcal{P}_{\mu}$.
\end{proof}
\begin{lem}\label{Lip}
There exists an unique minimizer $a^*(x,p,u)$ in \eqref{ham}, and $a^*(x,p,u)$ is continuos in $u$ and Lipschitz in $p$.
\end{lem}
\begin{proof}
Assumption (\ref{convex}) implies that the function $L(x,a,u)$ is strictly convex and coercive in the variable $a$, thus equation (\ref{ham}) has  an unique minimizer. 

Now we prove continuity in $u$. Using (\ref{convex}) and the fact that $a^*$ is a minimizer,
we have:
\[
L(x,a^*(x,p,u),u')+\sum_{y\neq x}a_y^*(x,p,u){p}_y\geq L(x,a^*(x,p,u'),u') +\sum_{y\neq x} a_y^*(x,p,u'){p}_y\geq
\]
\[
\geq
L(x,a^*(x,p,u),u') +  \nabla_a L(x,a^*(x,p,u),u') \cdot (a^*(x,{p},u')-a^*(x,p,u)) 
\]
\[
+\g \|a^*(x,{p},u')-a^*(x,p,u)\|^2
+ \sum_{y\neq x}a_y^*(x,{p},u') {p}_y
\]
From this we have:
\begin{equation*}
\begin{split}
\g \|a^*(x,{p},u')-& a^*(x,p,u)\|^2  \leq \left( p^x+\nabla_a L(x,a^*(x,p,u),u') \right) \cdot (a^*(x,{p},u)-a^*(x,p,u')) \\
& = \left( p^x+\nabla_a L(x,a^*(x,p,u),u) \right) \cdot (a^*(x,{p},u)-a^*(x,p,u')) \\
& + \left( \nabla_a L(x,a^*(x,p,u),u')- \nabla_a L(x,a^*(x,p,u),u)\right) \cdot(a^*(x,{p},u)-a^*(x,p,u')),
\end{split}
\end{equation*}
where $p^x_y = p_y$ for $y \neq x$, and $p^x_x=0$.
Since $a^*(x,{p},u)$ is the minimum of the function $\sum_{y \neq x} a_y p_y + L(x,a,u)$,  we have $p_y^x+\frac{\partial}{\partial a_y} L(x,a^*(x,p,u),u)=0$ except possibly if $a^*_y(x,p,u) = 0$, and in this case it must be $p_y^x+\frac{\partial}{\partial a_y} L(x,a^*(x,p,u),u)\geq 0$. Thus
\[
\left( p^x+\nabla_a L(x,a^*(x,p,u),u) \right) \cdot (a^*(x,{p},u)-a^*(x,p,u')) \leq 0.
\]
So we can rewrite the previous inequality as:
\[
\g \|a^*(x,{p},u')-a^*(x,p,u)\|^2\leq
\]
\[
\leq  \left( \nabla_a L(x,a^*(x,p,u),u')- \nabla_a L(x,a^*(x,p,u),u)\right) \cdot (a^*(x,{p},u)-a^*(x,p,u')) 
\]
Thus, we get:
\[
\|a^*(x,{p},u')-a^*(x,p,u)\|\leq \frac{1}{\gamma}\|\nabla_a L(x,a^*(x,p,u),u')- \nabla_a L(x,a^*(x,p,u),u)\|
\]
which  proves the continuity of $a^*$  in the $u$ variable, since by Assumption 1 $\nabla_a L(x,a,u)$  is continuous in $u$.

Now we prove that  $a^*(x,p,u)$ is  Lipschitz in $p$.
Since $a^*$ is a minimizer, and from (\ref{convex}),  we have:
\[
L(x,a^*(x,p,u),u)+\sum_{y\neq x}a_y^*(x,p,u)\tilde{p}_y\geq L(x,a^*(x,\tilde{p},u),u) +\sum_{y\neq x} a_y^*(x,\tilde{p},u)\tilde{p}_y\geq
\]
\[
\geq
L(x,a^*(x,p,u),u) +  \nabla_a L(x,a^*(x,p,u),u) \cdot (a^*(x,\tilde{p},u)-a^*(x,p,u)) 
\]
\[
+\g \|a^*(x,\tilde{p},u)-a^*(x,p,u)\|^2
+ \sum_{y\neq x}a_y^*(x,\tilde{p},u) \tilde{p}_y\]
Thus:
\[
\sum_{y\neq x}\left[ \left( a_y^*(x,\tilde{p},u) -a_y^*(x,p,u) \right)\tilde{p}_y\right]+\g \|a^*(x,\tilde{p},u)-a^*(x,p,u)\|^2
\]
\[
+\nabla_a L(x,a^*(x,p,u),u) \cdot (a^*(x,\tilde{p},u)-a^*(x,p,u))\leq 0\]

Adding and subtracting $\sum_{y\neq x}\left[ \left( a_y^*(x,\tilde{p},u) -a_y^*(x,p,u) \right){p}_y\right]$ we get:
\[
\sum_{y\neq x}\left[ \left( a_y^*(x,\tilde{p},u) -a_y^*(x,p,u) \right)\left(\tilde{p}_y-p_y\right)\right]+\g \|a^*(x,\tilde{p},u)-a^*(x,p,u)\|^2
\]
\[
+\left(\nabla_a L(x,a^*(x,p,u),u) + p^x\right) \cdot (a^*(x,\tilde{p},u)-a^*(x,p,u))\leq 0\]

Since $a^*$ is a minimizer we have:

\[
( \nabla_a L(x,a^*(x,p,u),u)  +p^x)\cdot (a-a^*(x,p,u))\geq 0.
\]
Thus 
\[
\g \|a^*(x,\tilde{p},u)-a^*(x,p,u)\|^2\leq -\sum_{y\neq x}\left[ \left( a_y^*(x,\tilde{p},u) -a_y^*(x,p,u) \right)\left(\tilde{p}_y-p_y\right)\right]\leq 
\]
\[
\leq\|\tilde{p}-p\|\|a^*(x,\tilde{p},u)-a^*(x,p,u)\|
\]
So 
\[
\|a^*(x,\tilde{p},u)-a^*(x,p,u)\|\leq \frac{1}{\g}\|\tilde{p}-p\|
\]
\end{proof}

Now we prove Theorem \ref{th:existence}.

\begin{proof}
Fix  $m\in \mathcal{C}\left([0,T],  \mathcal{P}_{\mu}\right)$, thus $m(t)(x,du) = m_x(t,u) \mu(du)$. Let $V^m(t,x,u)$ the solution of the first two equations in  \eqref{MFG}.  For this function $V^m$ consider the equation:
\be \label{due}
\begin{split}
\frac{d}{dt} n_x(t,u) & = \sum_{y \in \S} n_y(t,u) a_x^*(y,\nabla V^m(t,y,u),u) \\
n_x(0,u) & = m_x(0,u)
\end{split}
\ee
Given the solution of \eqref{due}, define $\Phi(m)(t)(x,du)=n_x(t,u)\mu(du)$.
Since the function $V^m(t,x,u)$ is continuous as a function of $m$ on $\mathcal{P}_{\mu}$, by Lemma \ref{V-continua}, we have that the function $\Phi$:
 \[
\begin{array}{lccc}
\Phi: & \mathcal{C}\left([0,T],  \mathcal{P}_{\mu}\right) &\to &\mathcal{C}\left([0,T],  \mathcal{P}_{\mu}\right) \\
& m &\mapsto &\Phi(m)
\end{array}
\]
is continuos. Here $\mathcal{C}\left([0,T],  \mathcal{P}_{\mu}\right)$ is considered as a metric space with, for $m_1,m_2 \in \mathcal{C}\left([0,T],  \mathcal{P}_{\mu}\right)$
\[
\mbox{dist}(m_1,m_2) := \sup_{t \in [0,T]} d(m_1(t),m_2(t)),
\]
where $d(\cdot,\cdot)$ is any metric inducing the topology of weak convergence.
We want to prove that $t \mapsto \Phi(m)(t)$ is Lipschitz, with the Lipschitz constant that does not depend on $m$. 

Notice that we have:
\[
d\left(\Phi(m)(t),\Phi(m)(s)\right)\leq \sum_{x\in \S}\int_{\mathcal U}|n_x(t,u)-n_x(s,u)|\mu(du)
\]
\[
= \sum_{x\in \S}\int_{\mathcal U} \left|\int_s^t \left(\frac{d}{dr}n_x(r,u)\right)dr\right| \mu(du)\leq
 \sum_{x\in \S}\int_{\mathcal U} \int_s^t \left| \frac{d}{dr}n_x(r,u)\right|  dr\mu(du)
\]
We know that $a^*(x,p,u)$ is continuos  in   $u\in \mathcal{U}$ and  $\mathcal{U}$ is compact. Moreover $a^*(x,p,u)$ is  Lipschitz in $p$ (see lemma \ref{Lip}), thus, since $V^m$ is bounded, $a^*(x,p,u)$ is  bounded in the range of $V^m$. So  there exists a  constant $C$ such that:
\[
 \sum_{x\in \S}\int_{\mathcal U} \int_s^t \left| \frac{d}{dr}n_x(r,u)\right|  dr\mu(du)
\leq C  \sum_{x\in \S}\int_{\mathcal U} \int_s^t    \sum_{y\in \S} n_y(r,u) dr\mu(du)
\leq C|\S| |t-s|
\]
So  for all measure  $m\in \mathcal{C}\left([0,T],  \mathcal{P}_{\mu}\right)$, 
we have
\[
d\left(\Phi(m)(t),\Phi(m)(s)\right)\leq C|\S| |t-s|
\]
Thus  we may restrict the map $\Phi$, to the subset  $\mathcal{H}$ of $ \mathcal{C}\left([0,T],  \mathcal{P}_{\mu}\right)$ of all flows of measures $m(t)(x,du)=m_x(t,u) \mu(du)$ which are   Lipschitz  with  constant $ C|\S| $. We have seen that $\Phi \left(\mathcal{H}\right)\subset \mathcal{H}$; moreover  $\mathcal{H}$ is convex and, by Ascoli-Arzela theorem, is compact. Thus by the Brower fixed point Theorem there exists $\bar{m}\in \mathcal{H}$ such that $\Phi(\bar{m})=\bar{m}$. It follows that the two functions $V^{\bar{m}}(t,x,u)$ and $\bar{m}$ are solutions of  the Mean Field Game equations \eqref{MFG}.
\end{proof}

\subsection{Proof of Theorem \ref{th:uniqueness}}

Assume that the two pairs $\Big(V(t,x,u), \, m(t)(x,du)\Big)$ and $\left(\tilde{V}(t,x,u), \, \tilde{m}(t)(x,du)\right)$, with 
$m(t)(x,du)=m_x(t,u) \mu(du)$ and $\tilde{m}(t)(x,du)=\tilde{m}_x(t,u) \mu(du)$,  
 are both solutions of equations \eqref{MFG}.
Let's consider the following derivative:
\be \label{derivata}
A=\frac{d}{dt}\left[\int_{\mathcal{U}}\sum_{x\in \S} \left(\tilde{V}(t,x,u)-V(t,x,u)\right)\left( \tilde{m}_x(t,u)-m_x(t,u)\right) \mu(du)\right]
\ee
We may interchange the derivate with the integral, and using equations \eqref{MFG}, we get:
\[
\int_{\mathcal{U}}\sum_{x\in \S}\left(H(x, \nabla\tilde{V}(t,x,u),u)-H(x, \nabla{V}(t,x,u),u)\right)\left( \tilde{m}_x(t,u)-m_x(t,u)\right)\mu(du)
\]
\[
-\int_{\mathcal{U}}\sum_{x\in \S}\left(F(x, \tilde{m}_x(t,u),u)-F(x, {m}_x(t,u),u)\right)\left( \tilde{m}_x(t,u)-m_x(t,u)\right)\mu(du)
\]
\[+
\int_{\mathcal{U}}\sum_{x\in \S} (\tilde{V}(t,x,u)-V(t,x,u))\left(\sum_{y\in \S}\tilde{m}_y(t,u)a^*_x(y, \nabla\tilde{V}(t,y,u),u)\right)\mu(du)
\]
\[
-
\int_{\mathcal{U}}\sum_{x\in \S} (\tilde{V}(t,x,u)-V(t,x,u))\left(\sum_{y\in \S}{m}_y(t,u)a^*_x(y, \nabla{V}(t,y,u),u)\right)\mu(du)
\]
 Assumption 3 implies that the second term of previous equation is negative. Using the definition of $H(x,p,u)$ we have that:
 \[
 A\leq 
 \]
 \[
\int_{\mathcal{U}}\sum_{x\in \S} \left(L(x,a^*(x,\nabla V{(t,x,u)},u)-L(x,a^*(x,\nabla \tilde{V}{(t,x,u)},u)\right)
\left( \tilde{m}_x(t,u)-m_x(t,u)\right)\mu(du)
\]
\[
+
\int_{\mathcal{U}}\sum_{x,y\in \S} \left[ a^*_y(x,\nabla V{(t,x,u)},u)\Big(V{(t,y,u)}-V{(t,x,u)}\Big)
\right]\left( \tilde{m}_x(t,u)-m_x(t,u)\right)\mu(du)
\]
\[
-
\int_{\mathcal{U}}\sum_{x,y\in \S} \left[ 
a^*_y(x,\nabla \tilde{V}(t,x,u),u)\Big(\tilde{V}(t,y,u)-\tilde{V}(t,x,u)\Big)\right]\left( \tilde{m}_x(t,u)-m_x(t,u)\right)\mu(du)
\]
\[
+
\int_{\mathcal{U}}\sum_{x,y\in \S}\Big(\tilde{V}(t,y,u)-{V}(t,y,u)\Big) \tilde{m}_x(t,u)a^*_y(x,\nabla\tilde{V}(t,x,u),u) \mu(du)
\]
\[
-
\int_{\mathcal{U}}\sum_{x,y\in \S} (\tilde{V}(t,x,u)-V(t,x,u)){m}_y(t,u)a^*_x(y, \nabla{V}(t,y,u),u)\mu(du)
\]
Since we have $\sum_{y\in \S} a^*_y(x,\nabla V(t,x,u),u)=0$ (same for $\tilde{V}$), we get:
 \[
 A\leq 
 \]
 \[
\int_{\mathcal{U}}\sum_{x\in \S} \left(L(x,a^*(x,\nabla V{(t,x,u)},u)-L(x,a^*(x,\nabla \tilde{V}{(t,x,u)},u)\right)
\left( \tilde{m}_x(t,u)-m_x(t,u)\right)\mu(du)
\]
\[
+
\int_{\mathcal{U}}\sum_{x,y\in \S}\left[a^*_y(x,\nabla V(t,x,u),u)V{(t,y,u)}\tilde{m}_x(t,u)\right] \mu(du)
\]
\[
+
\int_{\mathcal{U}}\sum_{x,y\in \S}\left[a^*_y(x,\nabla \tilde{V}(t,x,u),u)\tilde{V}{(t,y,u)} {m}_x(t,u)\right] \mu(du)
\]
\[
-
\int_{\mathcal{U}}\sum_{x,y\in \S}\left[ \tilde{V}{(t,y,u)}m_x(t,u)a^*_y(x,\nabla V(t,x,u),u)\right] \mu(du) 
\]
\[-
 \int_{\mathcal{U}}\sum_{x,y\in \S}\left[ 
{V}{(t,y,u)}\tilde{m}_x(t,u)a^*_y(x,\nabla \tilde{V}(t,x,u),u)\right] \mu(du)
\]
Thus:
 \[
 A\leq 
 \]
 \[
\int_{\mathcal{U}}\sum_{x\in \S} \left(L(x,a^*(x,\nabla V{(t,x,u)},u)-L(x,a^*(x,\nabla \tilde{V}{(t,x,u)},u)\right)
\left( \tilde{m}_x(t,u)-m_x(t,u)\right)\mu(du)
\]
\begin{equation} \label{abcd}
+\int_{\mathcal{U}}\sum_{x,y\in \S}\left[a^*_y(x,\nabla V(t,x,u),u)- a^*_y(x,\nabla \tilde{V}(t,x,u),u)\right]V{(t,y,u)}\tilde{m}_x(t,u) \mu(du)
\end{equation}
\[
+
\int_{\mathcal{U}}\sum_{x,y\in \S}\left[a^*_y(x,\nabla \tilde{V}(t,x,u),u)- a^*_y(x,\nabla{V}(t,x,u),u)\right]\tilde{V}{(t,y,u)}{m}_x(t,u) \mu(du).
\]
We begin to deal with the first term in the r.h.s. of \eqref{abcd}:
 \[
\int_{\mathcal{U}}\sum_{x\in \S} \left(L(x,a^*(x,\nabla V{(t,x,u)},u)-L(x,a^*(x,\nabla \tilde{V}{(t,x,u)},u)\right)
\left( \tilde{m}_x(t,u)-m_x(t,u)\right)\mu(du)
\]
\begin{equation}\label{bcde}
\begin{split}
=\int_{\mathcal{U}}\sum_{x\in \S} \left(L(x,a^*(x,\nabla V{(t,x,u)},u)-L(x,a^*(x,\nabla \tilde{V}{(t,x,u)},u)\right)
 \tilde{m}_x(t,u)\mu(du) \\
 + \int_{\mathcal{U}}\sum_{x\in \S} \left(L(x,a^*(x,\nabla \tilde{V}{(t,x,u)},u)-L(x,a^*(x,\nabla V{(t,x,u)},u)\right)
m_x(t,u)\mu(du).
\end{split}
\end{equation}
Using inequality (\ref{convex})
\begin{equation*}
\begin{split}
L(x,a^*(x,\nabla &V{(t,x,u)},u)  -L(x,a^*(x,\nabla \tilde{V}{(t,x,u)},u) \\ &\leq - \nabla_a L(x,\nabla V(t,x,u),u) \cdot \left(a^*(x,\nabla \tilde{V}(t,x,u),u)- a^*(x,\nabla V(t,x,u),u) \right)  \\ & ~~~- \gamma \|a^*(x,\nabla \tilde{V}(t,x,u),u)- a^*(x,\nabla V(t,x,u),u)\|^2
\end{split}
\end{equation*}
By definition of $a^*$, we have that 
\[
\frac{\partial}{\partial a_y} L(x,\nabla V(t,x,u),u) = V(t,y,u) - V(t,x,u)
\] 
unless $a_y^*(x,\nabla \tilde{V}(t,x,u),u)=0$ and in this case
\[
\frac{\partial}{\partial a_y} L(x,\nabla V(t,x,u),u) \geq V(t,y,u) - V(t,x,u).
\]
It follows that
\begin{equation*}
\begin{split}
- \nabla_a L(x,\nabla & V(t,x,u),u) \cdot \left(a^*(x,\nabla \tilde{V}(t,x,u),u)- a^*(x,\nabla V(t,x,u),u) \right) \\ & \leq \sum_{y \in \S} \left[V(t,y,u) - V(t,x,u)\right] \left(a_y^*(x,\nabla \tilde{V}(t,x,u),u)- a_y^*(x,\nabla V(t,x,u),u) \right).
\end{split}
\end{equation*}
This implies
{\small
\begin{equation*}
\begin{split}
\int_{\mathcal{U}}&\sum_{x\in \S} \left(L(x,a^*(x,\nabla V{(t,x,u)},u)-L(x,a^*(x,\nabla \tilde{V}{(t,x,u)},u)\right)
 \tilde{m}_x(t,u)\mu(du) \\ & \leq   -\int_{\mathcal{U}}\sum_{x,y\in \S} (V(t,y,u)-V(t,x,u))\left(a_y^*(x,\nabla V(t,x,u),u)-a_y^*(x,\nabla \tilde{V}(t,x,u),u)\right)\tilde{m}_x(t,u)\mu(du) \\ &  ~~~-\g \int_{\mathcal{U}}\sum_{x \in \S}  \|a^*(x,\nabla V(t,x,u),u)-a^*(x,\nabla \tilde{V}(t,x,u),u)\|^2 \tilde{m}_x(t,u)\mu(du) \\ & = -\int_{\mathcal{U}}\sum_{x,y\in \S} V(t,y,u)\left(a_y^*(x,\nabla V(t,x,u),u)-a_y^*(x,\nabla \tilde{V}(t,x,u),u)\right)\tilde{m}_x(t,u)\mu(du) \\ &  ~~~-\g \int_{\mathcal{U}}\sum_{x \in \S}  \|a^*(x,\nabla V(t,x,u),u)-a^*(x,\nabla \tilde{V}(t,x,u),u)\|^2 \tilde{m}_x(t,u)\mu(du),
 \end{split}
 \end{equation*}
 }
where in the last step we have used the fact that $\sum_y a^*_y = 0$.
In the same way one proves that
{\small
\begin{equation*}
\begin{split}
\int_{\mathcal{U}}& \sum_{x\in \S} \left(L(x,a^*(x,\nabla \tilde{V}{(t,x,u)},u)-L(x,a^*(x,\nabla V{(t,x,u)},u)\right)
m_x(t,u)\mu(du) \\ & \leq -\int_{\mathcal{U}}\sum_{x,y\in \S} \tilde{V}(t,y,u)\left(a_y^*(x,\nabla \tilde{V}(t,x,u),u)-a_y^*(x,\nabla {V}(t,x,u),u)\right){m}_x(t,u)\mu(du) \\ & -\g \int_{\mathcal{U}}\sum_{x \in \S}  \|a^*(x,\nabla V(t,x,u),u)-a^*(x,\nabla \tilde{V}(t,x,u),u)\|^2 m_x(t,u)\mu(du)
 \end{split}
 \end{equation*}
 }

Inserting these two inequalities in \eqref{bcde}, we obtain
{\small{
 \[
 A\leq 
 \]
 \[
-\int_{\mathcal{U}}\sum_{x,y\in \S} V(t,y,u)\left(a_y^*(x,\nabla V(t,x,u),u)-a_y^*(x,\nabla \tilde{V}(t,x,u),u)\right)\tilde{m}_x(t,u)\mu(du)
\]
\[
-\int_{\mathcal{U}}\sum_{x,y\in \S} \tilde{V}(t,y,u)\left(a_y^*(x,\nabla \tilde{V}(t,x,u),u)-a_y^*(x,\nabla {V}(t,x,u),u)\right){m}_x(t,u)\mu(du)
\]
\[
-\g \int_{\mathcal{U}}\sum_{x \in \S}  \|a^*(x,\nabla V(t,x,u),u)-a^*(x,\nabla \tilde{V}(t,x,u),u)\|^2\left(m_x(t,u)+\tilde{m}_x(t,u)\right)\mu(du)
\]
\[
+\int_{\mathcal{U}}\sum_{x,y\in \S}\left[a^*_y(x,\nabla V(t,x,u),u)- a^*_y(x,\nabla \tilde{V}(t,x,u),u)\right]V{(t,y,u)}\tilde{m}_x(t,u) \mu(du)
\]
\[
+
\int_{\mathcal{U}}\sum_{x,y\in \S}\left[a^*_y(x,\nabla \tilde{V}(t,x,u),u)- a^*_y(x,\nabla{V}(t,x,u),u)\right]\tilde{V}{(t,y,u)}{m}_x(t,u) \mu(du).
\]
\[
= -\g \int_{\mathcal{U}}\sum_{x \in \S}  \|a^*(x,\nabla V(t,x,u),u)-a^*(x,\nabla \tilde{V}(t,x,u),u)\|^2\left(m_x(t,u)+\tilde{m}_x(t,u)\right)\mu(du).
\]
}}
In conclusion, we get:
\[
\frac{d}{dt}\left[\int_{\mathcal{U}}\sum_{x\in \S} \left(\tilde{V}(t,x,u)-V(t,x,u)\right)\left( \tilde{m}_x(t,u)-m_x(t,u)\right) \mu(du)\right] \leq
\]
\[
\leq -\g \int_{\mathcal{U}}\sum_{x \in \S}  \|a^*(x,\nabla V(t,x,u),u)-a^*(x,\nabla \tilde{V}(t,x,u),u)\|^2\left(m_x(t,u)+\tilde{m}_x(t,u)\right)\mu(du)
\]
By taking the integral between $0$ and $T$, and since $m$ and  $\tilde{m}$ have the same initial condition, we have:
\[
\int_{\mathcal{U}}\sum_{x\in \S}
\left( G(x,\tilde{m}_x(T,u),u)- G(x,{m}_x(T,u),u)\right)\left(\tilde{m}_x(T,u)-m_{x}(T,u)\right)\mu(du) \leq
\]
\[
\leq
-\g\int_0^T \int_{\mathcal{U}}\sum_{x \in \S}  \|a^*(x,\nabla V(t,x,u),u)-a^*(x,\nabla \tilde{V}(t,x,u),u)\|^2\left(m_x(t,u)+\tilde{m}_x(t,u)\right)\mu(du)dt
\]
By Assumption 3, the first term of previous inequality is non negative, this implies that:
\[
\int_0^T \int_{\mathcal{U}}\sum_{x \in \S}  \| a^*(x,\nabla V(t,x,u),u)-a^*(x,\nabla \tilde{V}(t,x,u),u)\|^2\left(m_x(t,u)+\tilde{m}_x(t,u)\right)\mu(du)dt=0
\]
Thus:
\begin{equation} \label{uguale}
m_x(t,u)+\tilde{m}_x(t,u)>0 \ \ \Rightarrow \ \ a^*(x,\nabla V(t,x,u),u)=a^*(x,\nabla \tilde{V}(t,x,u),u)
\end{equation}
Equation \ref{uguale} implies that $m$ and $\tilde{m}$ are both solutions of the same ordinary differential equation,  so they are equal by the uniqueness theorem for o.d.e. 
This, in turn, implies that also $V$ and $\tilde{V}$ are solution of the same o.d.e, so  also these two functions are equal.

\subsection{Proof of Theorem \ref{th:approx}}

Let $\mathcal{B}$ be the set of predictable strategies, as defined in Remark \ref{rem:openloop}. For $\beta \in \mathcal{B}$,
we recall (see \eqref{iplayercost} and \eqref{mfcost}) that
{\footnotesize
\[
J_i^N([\boldsymbol{\alpha}^{N,-i}; \beta]) = \E \left[ \int_0^T \left[L(X^N_i(t), \beta(t), u_i^N) + F_N(X_i^N(t), m_{{\bs X}}^{N,i}(t), u_i^N) \right] dt + G_N(X_i^N(T), m_{{\bs X}}^{N,i}(T), u_i^N) \right].
\]
and
\[
J_{u_i^N}(\beta,m) = \E \left[ \int_0^T \left[L(Y^{\beta}_{u_i^N}(t), \beta(t), u_i^N) + F(Y^{\beta}_{u_i^N}(t), m(t), u_i^N) \right] dt + G(Y^{\beta}_{u_i^N}(T), m(T), u_i^N) \right].
\]
}
Note that:
\smallskip
\begin{itemize}
\item
$X_i^N(t)$ and $Y^{\beta}_{u_i^N}(t)$ evolve with the same jump intensities $\beta$ and have the same initial condition, so they have the same law;
\item
the processes $(X_j^N(t))_{j \neq i}$ are independent, and $X_j^N(t)$ has distribution $m(t,u_j^N)$.
\end{itemize}
Set 
{\footnotesize
\[
\tilde{J}_i^N([\boldsymbol{\alpha}^{N,-i}; \beta]) = \E \left[ \int_0^T \left[L(X^N_i(t), \beta(t), u_i^N) + F(X_i^N(t), m_{{\bs X}}^{N,i}(t), u_i^N) \right] dt + G(X_i^N(T), m_{{\bs X}}^{N,i}(T), u_i^N) \right],
\]
}
obtained by $J_i^N([\boldsymbol{\alpha}^{N,-i}; \beta])$ by replacing $F_N,G_N$ with $F,G$. We begin by proving
\begin{equation} \label{approxtilde}
\lim_{N \ra +\infty} \sup_{i=1,\ldots,N} \sup_{\beta \in \mathcal{B}} \left| \tilde{J}_i^N([{\bs \a}^{N,-i}, \beta]) - J_{u_i^N}(\beta,m) \right| = 0,
\end{equation}
By dominated convergence this follows if we show that for every $t \in [0,T)$
\begin{equation} \label{approx3}
\lim_{N \ra +\infty} \sup_{i=1,\ldots,N} \sup_{\beta \in \mathcal{B}} \E \left[\left|F(X_i^N(t), m_{{\bs X}}^{N,i}(t), u_i^N) - F(X_i^N(t), m(t), u_i^N) \right|\right] = 0,
\end{equation}
and the same for $t = T$ with $F$ replaced by $G$.
To show this, we fix an arbitrary $\e>0$. By continuity of $F$, there is a neighborhood $U$ of $m(t)$ (in weak topology) such that for all $\rho \in U$, all $x \in \Sigma$ and all $u \in \mathcal{U}$
\[
|F(x,\rho,u) - F(x, m(t),u)| \leq \e.
\]
Without loss of generality, $U$ can be taken of the form
\[
\left\{ \rho: \, \left|\int \phi_k d\rho - \int \phi_k dm(t) \right| < \delta, \, k=1,\ldots,m \right\},
\]
where $\phi_k: \Sigma \times \mathcal{U} \ra \R$ are bounded continuous functions and $\delta>0$. We now show that
\begin{equation} \label{intorno}
\lim_{N \ra +\infty} \P\left(m_{{\bs X}}^{N,i}(t) \in U\right) = 1.
\end{equation}
By Markov inequality, this follows if we show that, for each $k=1,\ldots,M$
\begin{equation} \label{intorno1}
\lim_{N \ra +\infty} \E \left[ \left|\int \phi_k d m_{{\bs X}}^{N,i}(t) - \int \phi_k dm(t) \right| \right] = 0.
\end{equation}
We do this in two steps. First
note that, using independence of the $X_j^N$,
\begin{equation} \label{lln1}
\begin{split}
\E& \left[\left| \int \phi_k d m_{{\bs X}}^{N,i}(t) -  \frac{1}{N-1} \sum_{j \neq i} \E[\phi_k(X_j^N(t),u_j^N)] \right| \right]
 \\ & ~~~~~= \E\left[\left| \frac{1}{N-1} \sum_{j \neq i}\phi_k(X_j^N(t),u_j^N) -  \frac{1}{N-1} \sum_{j \neq i} \E[\phi_k(X_j^N(t),u_j^N)] \right| \right] \\ & ~~~~~ \leq \left( \mbox{Var} \left( \frac{1}{N-1} \sum_{j \neq i}\phi_k(X_j^N(t),u_j^N)  \right) \right)^{1/2} \\ & ~~~~~  \leq \frac{C}{\sqrt{N}}
 \end{split}
\end{equation}
for a constant $C>0$ which does not depend on $i$ or $k$. Moreover, note now that 
\[
 \E[\phi_k(X_j^N(t),u_j^N)]  = \sum_{x \in \Sigma} \phi_k(x,u_j^N) m_x(t,u_j^N) =: \psi(t,u_j^N),
 \]
 where
 \[
 \psi(t,u) = \sum_{x \in \Sigma} \phi_k(x,u) m_x(t,u) 
 \]
 is continuous and bounded in $u$. Thus, by \eqref{limpos}, 
 \begin{equation} \label{lln2}
 \lim_{N \ra +\infty} \frac{1}{N-1} \sum_{j \neq i} \E[\phi_k(X_j^N(t),u_j^N)]  = \int \psi(t,u) d\mu = \int \phi_k dm(t).
\end{equation}
From \eqref{lln1} and \eqref{lln1},  \eqref{intorno1}, and therefore \eqref{intorno}.
By \eqref{intorno}, it follows that there exists $\bar{N}$, independent of $i$ and $\beta$, such that for $N>\bar{N}$, 
\[
\P\left(m_{{\bs X}}^{N,i}(t) \in U\right) \geq 1-\e.
\]
Therefore for $N \geq \bar{N}$
\begin{equation*}
\begin{split}
 \E & \left[\left|F(X_i^N(t), m_{{\bs X}}^{N,i}(t), u_i^N)  - F(X_i^N(t), m(t), u_i^N) \right|\right]  \\ 
& ~~~~~~ \leq \e + 2 \|F\|_{\infty} \P\left(m_{{\bs X}}^{N,i}(t) \not\in U\right) \leq (1+2|F\|_{\infty} )\e,
\end{split} 
\end{equation*}
which completes the proof of \eqref{approx3}, and therefore of \eqref{approxtilde}.

We now compare $\tilde{J}_i^N([\boldsymbol{\alpha}^{N,-i}; \beta])$ with $J_i^N([\boldsymbol{\alpha}^{N,-i}; \beta])$. Using dominated convergence, it is easily shown that, under Assumption 5(a)
\begin{equation} \label{approxJ}
\lim_{N \ra +\infty} \sup_{i=1,\ldots,N} \sup_{\beta \in \mathcal{B}} \left| \tilde{J}_i^N([{\bs \a}^{N,-i}, \beta]) - J_i^N([\boldsymbol{\alpha}^{N,-i}; \beta]) \right| = 0,
\end{equation}
while, under Assumption 5(b)
\begin{equation} \label{approxJb}
\lim_{N \ra +\infty} \frac{1}{N} \sum_{i=1}^N \sup_{\beta \in \mathcal{B}} \left| \tilde{J}_i^N([{\bs \a}^{N,-i}, \beta]) - J_i^N([\boldsymbol{\alpha}^{N,-i}; \beta]) \right| = 0,
\end{equation}
We now complete the proof for the case Assumption 5(b) holds, and we will indicate where Assumption 5(a) allows to obtain the stronger result. Define
\begin{equation} \label{eN'}
\e'_N := \frac{1}{N} \sum_{i=1}^N \sup_{\beta \in \mathcal{B}} \left| \tilde{J}_i^N([{\bs \a}^{N,-i}, \beta]) - J_i^N([\boldsymbol{\alpha}^{N,-i}; \beta]) \right|.
\end{equation}
By \eqref{approxJb}, $\e'_N \ra 0$ as $N \ra +\infty$. Let $k_N$ be any sequence such that $k_N \ra +\infty$ and $k_N \e_N \ra 0$. Set
\begin{equation} \label{kN}
A_N := \{ i : \sup_{\beta \in \mathcal{B}}\left| \tilde{J}_i^N([{\bs \a}^{N,-i}, \beta]) - J_i^N([\boldsymbol{\alpha}^{N,-i}; \beta]) \right| > k_N \e_N \}.
\end{equation}
Clearly, by \eqref{eN'}
\[
\e'_N \geq \frac{|A_N|}{N} k_N \e'_N,
\]
and therefore $\delta_N := \frac{|A_N|}{N} \ra 0$; note that under Assumption 5(a) we could define
\[
\e'_N := \sup_{i=1,\ldots,N} \sup_{\beta \in \mathcal{B}} \left| \tilde{J}_i^N([{\bs \a}^{N,-i}, \beta]) - J_i^N([\boldsymbol{\alpha}^{N,-i}; \beta]) \right|,
\]
obtaing
$A_N = \emptyset$, so $\delta_N = 0$. Now define
\[
\e''_N := \sup_{i=1,\ldots,N} \sup_{\beta \in \mathcal{B}} \left| \tilde{J}_i^N([{\bs \a}^{N,-i}, \beta]) - J_{u_i^N}(\beta,m) \right|  + k_N \e'_N,
\]
which goes to zero by \eqref{approxtilde}. Note that combining \eqref{approxtilde} and \eqref{kN}, we have that 
\[
\sup_{\beta \in \mathcal{B}} \left| J_i^N([{\bs \a}^{N,-i}, \beta]) - J_{u_i^N}(\beta,m) \right| \leq \e''_N
\]
for all $i \not\in A_N$. Setting $\e_N := 2 \e''_N$, and using the fact that $\a^*$ is the optimal control for the mean-field game, i.e.  $J_{u_i^N}(\a^*,m) \leq J_{u_i^N}(\beta,m)$ for all $\beta \in \mathcal{B}$, we have, again for all $\beta \in \mathcal{B}$ and all $i \not\in A_N$,
\[
J_i^N(\boldsymbol{\alpha}^N) \leq J_{u_i^N}(\a^*,m) + \frac12 \varepsilon_N \leq J_{u_i^N}(\beta,m)+ \frac12 \varepsilon_N  \leq J_i^N([\boldsymbol{\alpha}^{N,-i}; \beta]) + \varepsilon_N,
\]
completing the proof.

\end{document}